\documentclass[a4, 12pt]{amsart}
\usepackage[mathscr]{eucal}
\usepackage{amssymb}
\usepackage{latexsym}
\usepackage{amsthm}
\usepackage{color}
\theoremstyle{plain}
\newtheorem{theorem}{Theorem}[section]

\newtheorem{lemma}{Lemma}[section]
\newtheorem{proposition}{Proposition}[section]

\makeatletter
\@addtoreset{equation}{section}

\title[Complete lagrangian self-expanders in $\mathbb C^{2}$]
{Complete lagrangian self-expanders in $\mathbb C^{2}$}

\author [Z. Li and G. Wei]{Zhi Li and Guoxin Wei}

\address{Zhi Li \\  \newline \indent College of Mathematics and Information Science, Henan Normal University,
\newline \indent 453007, Xinxiang, Henan, China.}
\email{lizhihnsd@126.com}

\address{Guoxin Wei \\ \newline \indent School of Mathematical Sciences, South China Normal University,
\newline \indent 510631, Guangzhou,  China.}
\email{weiguoxin@tsinghua.org.cn}

\begin{document}
\maketitle

\begin{abstract}
In this paper, we obtain a classification theorem of $2$-dimensional complete Lagrangian self-expanders with constant squared norm of the second fundamental form in $\mathbb C^{2}$.
\end{abstract}


\section{introduction}
\vskip2mm
\noindent

An $n$-dimensional smooth immersed submanifold $x: M^{n}\to \mathbb{R}^{n+p}$ is called a self-expanders of the mean curvature flow if its mean curvature vector $\vec{H}$ satisfies the following non-linear elliptic equation
\begin{equation*}
\vec{H}=x^{\perp},
\end{equation*}
where where $x^{\perp}$ is the orthogonal projection of
the position vector $x$ in $\mathbb{R}^{n+p}$ to the normal bundle of $M^{n}$.
It is well known that if $x$ is a self-expanders then $x_{t}=\sqrt{2t}x, \ \ t>0$ is moved by the mean curvature flow.

Notice that the mean curvature flow always blows up at finite time. For
noncompact hypersurfaces, solution to the mean curvature flow may exist
for all times. The self-expanders appear as the singularity model
of the mean curvature flow which exists for long time. For the case of codimension one, motivation for the study of self-expanders goes back to work of Ecker-Huisken \cite{EH} and Stavrou \cite{Sta}. Specially, the rigidity theorem of the self-expanders were studied by many geometers. Specific references can be found in \cite{AC}, \cite{CZ}, \cite{H}, \cite{Ish}, \cite{Smo} and the references therein.

In addition, the properties of higher codimension self-expanders have been
studied for the Lagrangian mean curvature flow.
The first examples of Lagrangian self-expanders were constructed in Anciaux \cite{Anc} and
Lee and Wang (\cite{LW1}). In \cite{LN}, Lotay and Neves proved that zero-Maslov class Lagrangian self-expanders in $\mathbb{C}^{n}$ that are asymptotic to a pair of planes intersecting transversely are locally unique for $n>2$ and unique if $n=2$.
Later, Nakahara \cite{Nak} construct a smooth Lagrangian self-expander asymptotic to any pair of Lagrangian planes in $\mathbb{C}^{n}$ which
intersect transversely at the origin and have sum of characteristic angles less than $\pi$. In 2016, Imagi, Joyce and dos Santos \cite{IJS} shown further uniqueness results for Lagrangian self-expanders asymptotic to the union of two transverse Lagrangian planes.
More other results about the Lagrangian self-expanders have been done (see \cite{Cas}, \cite{GSSZ}, \cite{JLT}, \cite{LW}, \cite{Nev}).

Recently, Cheng, Hori and Wei \cite{CHW} established the
following classification theorem for complete Lagrangian self-shrinkers in $\mathbb{C}^{2}$:

\noindent
\begin{theorem}\label{theorem 1.1}
Let $x: M^{2}\to \mathbb{C}^{2}$ be a
$2$-dimensional complete Lagrangian self-shrinkers with the constant squared norm of the second fundamental form in $\mathbb{C}^{2}$,
then $x(M^{2})$  is one of the following surfaces:
$\mathbb {R}^{2}$, $\mathbb {S}^{1}(1)\times \mathbb {R}^{1}$ and $\mathbb {S}^{1}(1)\times \mathbb {S}^{1}(1)$.
\end{theorem}

\vskip2mm

Motivated by the above theorem, we aim in this paper to study the complete
Lagrangian self-expander in $\mathbb{C}^{2}$. As the result, we obtain the following result:

\begin{theorem}\label{theorem 1.2}
Let $x: M^{2}\to \mathbb C^{2}$ be a
$2$-dimensional complete Lagrangian self-expander with constant squared norm of the second fundamental form, then $x(M^{2})$ is a hyperplane
$\mathbb {R}^{2}$ through the origin.
\end{theorem}

\vskip5mm
\section {Preliminaries}
\vskip2mm

\noindent

Let $x: M^{2} \rightarrow\mathbb C^{2}$ be an
$2$-dimensional Lagrangian surface of $\mathbb C^{2}$. Denote
by $J$ the canonical complex structure on $\mathbb C^{2}$.
We choose orthonormal tangent vector fields $\{e_{1}, e_{2}\}$ and $\{e_{1^{\ast}}, e_{2^{\ast}}\}$ are normal vector fields given by
$$e_{1^{\ast}}=J e_{1}, \ e_{2^{\ast}}=Je_{2}.$$
Then
$$\{e_{1}, e_{2}, e_{1^{\ast}}, e_{2^{\ast}}\}$$
is called an adapted Lagrangian frame field. The dual frame fields of $\{e_{1}, e_{2}\}$ are $\{\omega_{1}, \omega_{2}\}$, the Levi-Civita connection forms and normal
connection forms are $\omega_{ij}$ and $\omega_{i^{\ast}j^{\ast}}$ , respectively.

Since $x: M^{2} \rightarrow\mathbb C^{2}$ is a Lagrangian surface (see \cite{LV}, \cite{LW2}), we have
\begin{equation}\label{2.1-1}
h^{p^{\ast}}_{ij}=h^{p^{\ast}}_{ji}=h^{i^{\ast}}_{pj}, \ \ i,j,p=1,2.
\end{equation}
The second fundamental form $h$ and the mean curvature $\vec{H}$ of $x$ are respectively
defined by $$h=\sum_{ijp}h^{p^{\ast}}_{ij}\omega_{i}\otimes\omega_{j}\otimes e_{p^{\ast}},\ \ \vec{H}=\sum_{p}H^{p^{\ast}}e_{p^{\ast}}=\sum_{i,p}h^{p^{\ast}}_{ii}e_{p^{\ast}}.$$
Let $S=\sum_{i,j,p}(h^{p^{\ast}}_{ij})^{2}$ be the squared norm of the second
fundamental form and $H=|\vec{H}|$ denote the mean curvature of $x$. If we denote the components of curvature tensors of the Levi-Civita connection forms $\omega_{ij}$ and normal connection forms $\omega_{i^{\ast}j^{\ast}}$ by
$R_{ijkl}$ and $R_{i^{\ast}j^{\ast}kl}$, respectively, then for $i, j, k, l, p, q=1, 2$, the equations of Gauss, Codazzi and Ricci are given by
\begin{equation}\label{2.1-2}
R_{ijkl}=\sum_{p}(h^{p^{\ast}}_{ik}h^{p^{\ast}}_{jl}-h^{p^{\ast}}_{il}h^{p^{\ast}}_{jk}),
\end{equation}
\begin{equation}\label{2.1-3}
R_{ik}=\sum_{p}H^{p^{\ast}}h^{p^{\ast}}_{ik}-\sum_{j,p}h^{p^{\ast}}_{ij}h^{p^{\ast}}_{jk},
\end{equation}
\begin{equation}\label{2.1-4}
h^{p^{\ast}}_{ijk}=h^{p^{\ast}}_{ikj},
\end{equation}
\begin{equation}\label{2.1-5}
R_{p^{\ast}q^{\ast}kl}=\sum_{i}(h^{p^{\ast}}_{ik}h^{q^{\ast}}_{il}
-h^{p^{\ast}}_{il}h^{p^{\ast}}_{ik}),
\end{equation}
\begin{equation}\label{2.1-6}
R=H^{2}-S.
\end{equation}

From \eqref{2.1-1} and \eqref{2.1-4}, we easily know that the components $h^{p^{\ast}}_{ijk}$ is totally symmetric for $i, j, k,l$. In particular,
\begin{equation}\label{2.1-7}
h^{p^{\ast}}_{ijk}=h^{p^{\ast}}_{kji}=h^{i^{\ast}}_{pjk}, \ \ i, j, k ,p=1, 2.
\end{equation}
By making use of \eqref{2.1-1}, \eqref{2.1-2} and \eqref{2.1-5}, we obtain
\begin{equation}\label{2.1-8}
R_{ijkl}=K(\delta_{ik}\delta_{jl}-\delta_{il}\delta_{jk})=R_{i^{\ast}j^{\ast}kl}, \ \ K=\frac{1}{2}(H^{2}-S),
\end{equation}
where $K$ is the Gaussian curvature of $x$.

\noindent By defining
\begin{equation*}
\sum_{l}h^{p^{\ast}}_{ijkl}\omega_{l}=dh^{p^{\ast}}_{ijk}+\sum_{l}h^{p^{\ast}}_{ljk}\omega_{li}
+\sum_{l}h^{p^{\ast}}_{ilk}\omega_{lj}+\sum_{l} h^{p^{\ast}}_{ijl}\omega_{lk}+\sum_{q} h^{q^{\ast}}_{ijk}\omega_{q^{\ast}p^{\ast}},
\end{equation*}
and
\begin{equation*}
\begin{aligned}
\sum_{m}h^{p^{\ast}}_{ijklm}\omega_{m}&=dh^{p^{\ast}}_{ijkl}+\sum_{m}h^{p^{\ast}}_{mjkl}\omega_{mi}
+\sum_{m}h^{p^{\ast}}_{imkl}\omega_{mj}+\sum_{m}h^{p^{\ast}}_{ijml}\omega_{mk}\\
&\ \ +\sum_{m}h^{p^{\ast}}_{ijkm}\omega_{ml}+\sum_{m}h^{m^{\ast}}_{ijkl}\omega_{m^{\ast}p^{\ast}},
\end{aligned}
\end{equation*}
we have the following Ricci identities
\begin{equation}\label{2.1-9}
h^{p^{\ast}}_{ijkl}-h^{p^{\ast}}_{ijlk}=\sum_{m}
h^{p^{\ast}}_{mj}R_{mikl}+\sum_{m} h^{p^{\ast}}_{im}R_{mjkl}+\sum_{m} h^{m^{\ast}}_{ij}R_{m^{\ast}p^{\ast}kl},
\end{equation}
and
\begin{equation}\label{2.1-10}
\begin{aligned}
h^{p^{\ast}}_{ijkln}-h^{p^{\ast}}_{ijknl}=&\sum_{m} h^{p^{\ast}}_{mjk}R_{miln}
+\sum_{m}h^{p^{\ast}}_{imk}R_{mjln}+ \sum_{m}h^{p^{\ast}}_{ijm}R_{mkln}\\
&+\sum_{m}h^{m^{\ast}}_{ijk}R_{m^{\ast}p^{\ast}ln}.
\end{aligned}
\end{equation}

Let $V$ be a tangent $C^{1}$-vector field on $M^{n}$ and denote by $Ric_{V} := Ric-\frac{1}{2}L_{V}g$ the
Bakry-Emery Ricci tensor with $L_{V}$ to be the Lie derivative along the vector field $V$. Define a
differential operator
\begin{equation*}
\mathcal{L}f=\Delta f+\langle V,\nabla f\rangle,
\end{equation*}
where $\Delta$ and $\nabla$ denote the Laplacian and the gradient
operator, respectively. The following maximum principle
of Omori-Yau type concerning the operator $\mathcal{L}$ will
be used in this paper, which was proved by Chen and Qiu \cite{CQ}.

\begin{lemma}\label{lemma 2.1}
Let $(M^{n}, g)$ be a complete Riemannian manifold, and $V$ is a $C^{1}$ vector field on $M^{n}$. If the Bakry-Emery Ricci tensor $Ric_{V}$ is bounded from below, then for any $f\in C^{2}(M^{n})$ bounded from above, there exists a sequence $\{p_{t}\} \subset M^{n}$, such that
\begin{equation*}
\lim_{m\rightarrow\infty} f(p_{t})=\sup f,\quad
\lim_{m\rightarrow\infty} |\nabla f|(p_{t})=0,\quad
\lim_{m\rightarrow\infty}\mathcal{L}f(p_{t})\leq 0.
\end{equation*}
\end{lemma}

For the mean curvature vector field $\vec{H}=\sum_{p}H^{p^{\ast}}e_{p^{\ast}}$, we define
\begin{equation}\label{2.1-11}
\begin{aligned}
|\nabla^{\perp}\vec{H}|^{2}=\sum_{i,p}(H^{p^{\ast}}_{,i})^{2}, \ \ \Delta^{\perp}H^{p^{\ast}}=\sum_{i}H^{p^{\ast}}_{,ii}.
\end{aligned}
\end{equation}
We next suppose that  $x: M^{2}\rightarrow\mathbb{C}^{2}$
is a lagrangian self-expander, that is, $H^{p^{\ast}}=\langle x, e_{p^{\ast}}\rangle$. By a simple calculation, we have the following basic formulas.

\begin{equation}\label{2.1-12}
\aligned
H^{p^{\ast}}_{,i}
=&-\sum_{k}h^{p^{\ast}}_{ik}\langle x, e_{k}\rangle, \\
H^{p^{\ast}}_{,ij}
=&-\sum_{k}h^{p^{\ast}}_{ijk}\langle x, e_{k}\rangle-h^{p^{\ast}}_{ij}-\sum_{k,q}h^{p^{\ast}}_{ik}h^{q^{\ast}}_{kj}H^{q^{\ast}},  \ \ i,j,p =1,2.
\endaligned
\end{equation}

\noindent
For $V=x^{\top}$, using the above formulas and the Ricci identities, we can get the following Lemmas. The specific calculation process is similar to \cite{CHW}.

\begin{lemma}\label{lemma 2.2}
Let $x:M^{2}\rightarrow \mathbb{C}^{2}$ be an $2$-dimensional complete lagrangian self-expander. We have
\begin{equation}\label{2.1-13}
\aligned
\frac{1}{2}\mathcal{L} H^{2}=&|\nabla^{\perp} \vec{H}|^{2}-H^{2}-\sum_{i,j,p,q}H^{p^{\ast}}h^{p^{\ast}}_{ij}H^{q^{\ast}}h^{q^{\ast}}_{ij}
\endaligned
\end{equation}
and
\begin{equation}\label{2.1-14}
\frac{1}{2}\mathcal{L}S
=\sum_{i,j,k}(h^{p^{\ast}}_{ijk})^{2}-S(\frac{3}{2}S+1)+2H^{2}S-\frac{1}{2}H^{4}
-\sum_{i,j,p,q}H^{p^{\ast}}h^{p^{\ast}}_{ij}H^{q^{\ast}}h^{q^{\ast}}_{ij}.
\end{equation}
\end{lemma}

\begin{lemma}\label{lemma 2.3}
Let $x:M^{2}\rightarrow \mathbb{C}^{2}$ be an $2$-dimensional complete lagrangian self-expander. If $S$ is constant, we infer that
\begin{equation}\label{2.1-15}
\aligned
&\frac{1}{2}\mathcal{L}\sum_{i,j,k}(h^{p^{\ast}}_{ijk})^{2}\\
=&\sum_{i,j,k,l,p}(h^{p^{\ast}}_{ijkl})^{2}+(10K-2)\sum_{i,j,k,p}(h^{p^{\ast}}_{ijk})^{2}-5K|\nabla^{\perp} \vec{H}|^{2}-\langle\nabla K, \nabla H^{2}\rangle\\
&-3\sum_{i,l,p}K_{,i}H^{p^{\ast}}_{,l}h^{p^{\ast}}_{il}
-2\sum_{i,j,k,l,p,q}h^{p^{\ast}}_{ijk}h^{p^{\ast}}_{ijl}h^{q^{\ast}}_{kl}H^{q^{\ast}}
-\sum_{i,j,k,l,p,q}h^{p^{\ast}}_{ijk}h^{p^{\ast}}_{il}h^{q^{\ast}}_{jl}H^{q^{\ast}}_{,k}\\
&-\sum_{i,j,k,l,p,q}h^{p^{\ast}}_{il}h^{p^{\ast}}_{ijk}h^{q^{\ast}}_{jkl}H^{q^{\ast}}
\endaligned
\end{equation}
and
\begin{equation}\label{2.1-16}
\aligned
&\frac{1}{2}\mathcal{L}\sum_{i,j,k}(h^{p^{\ast}}_{ijk})^{2}\\
=&(H^{2}-2S)\Big(|\nabla^{\perp} \vec{H}|^{2}-H^{2}\Big)+\frac{1}{2}|\nabla H^{2}|^{2}\\
&+(3K-2-H^{2}+2S)\sum_{j,k,p,q}H^{p^{\ast}}h^{p^{\ast}}_{jk}H^{q^{\ast}}h^{q^{\ast}}_{jk}\\
&-K\Big(H^{4}+\sum_{j,k,p}h^{p^{\ast}}_{jk}H^{p^{\ast}}H^{j^{\ast}}H^{k^{\ast}} \Big)
-\sum_{i,j,k,l,p,q}H^{p^{\ast}}h^{p^{\ast}}_{ij}h^{q^{\ast}}_{ij}h^{q^{\ast}}_{kl}h^{m^{\ast}}_{kl}H^{m^{\ast}}\\
&-\sum_{j,k,l,p,q}H^{p^{\ast}}h^{p^{\ast}}_{jk}H^{q^{\ast}}h^{q^{\ast}}_{jl}H^{m^{\ast}}h^{m^{\ast}}_{kl}
+2\sum_{i,j,k,p,q}H^{p^{\ast}}_{,i}h^{p^{\ast}}_{ijk}h^{q^{\ast}}_{jk}H^{q^{\ast}}\\
&+\sum_{i,j,k}\Big(\sum_{p}(H^{p^{\ast}}_{,i}h^{p^{\ast}}_{jk}+H^{p^{\ast}}h^{p^{\ast}}_{ijk})\Big)
\Big(\sum_{q}(H^{q^{\ast}}_{,i}h^{q^{\ast}}_{jk}+H^{q^{\ast}}h^{q^{\ast}}_{ijk})\Big).
\endaligned
\end{equation}
\end{lemma}

 \vskip10mm
\section{Proof of Theorem 1.1}

\vskip2mm

In order to use the maximum principle of Omori-Yau type (Lemma \ref{lemma 2.1}), we need the following proposition.
\begin{proposition}\label{proposition 3.1}
For a complete self-expander $x:M^{n}\rightarrow \mathbb{C}^{n}$ with non-zero constant squared norm $S$ of the second fundamental form, the Bakry-Emery Ricci tensor $Ric_{V}$ is bounded from below, where $V=x^{\top}$.
\end{proposition}
\begin{proof}
For any unit vector $e\in T M^{n}$, we can 
choose a local tangent orthonormal frame field $\{{e_{i}}\}^{n}_{i=1}$ such that $e=e_{i}$. Then, by the definition of self-expander, direct computation gives
\begin{equation*}
\begin{aligned}
\frac{1}{2}L_{x^{\top}}g(e_{i},e_{i})
=&\frac{1}{2}x^{\top}(g(e_{i},e_{i}))-g([x^{\top},e_{i}],e_{i})\\
=&g(\nabla_{e_{i}}(x-x^{\bot}),e_{i})\\
=&1-\sum_{p}H^{p^{\ast}}g(\nabla_{e_{i}}e_{p^{\ast}},e_{i})\\
=&1+\sum_{p}H^{p^{\ast}}h^{p^{\ast}}_{ii}.
\end{aligned}
\end{equation*}
Then \eqref{2.1-3} yields
\begin{equation*}
\begin{aligned}
Ric_{x^{\top}}(e_{i},e_{i})
=&Ric(e_{i},e_{i})-\frac{1}{2}L_{x^{\top}}g(e_{i},e_{i})\\
=&\sum_{p}H^{p^{\ast}}h^{p^{\ast}}_{ii}-\sum_{j,p}(h^{p^{\ast}}_{ij})^{2}-(1+\sum_{p}H^{p^{\ast}}h^{p^{\ast}}_{ii})\\
=&-\sum_{j,p}(h^{p^{\ast}}_{ij})^{2}-1\geq-(S+1).
\end{aligned}
\end{equation*}
The proof of the Proposition \ref{proposition 3.1} is finished.
\end{proof}

Next, using \eqref{2.1-12}--\eqref{2.1-16}, we obtain the following important proposition.

\begin{proposition}\label{proposition 3.2}
Let $x:M^{2}\rightarrow \mathbb{C}^{2}$
be a self-expander with the constant squared norm $S$ of the second fundamental form, then $H^{2}\equiv0$.
\end{proposition}
\begin{proof}
Since $S$ is constant, by taking exterior derivative of $S$ and using \eqref{2.1-14} to obtain that
\begin{equation}\label{3.1-1}
\begin{aligned}
&\sum_{i,j,p}h^{p^{\ast}}_{ij}h^{p^{\ast}}_{ijk}=0, \ \ \sum_{i,j,p}h^{p^{\ast}}_{ij}h^{p^{\ast}}_{ijkl}
+\sum_{i,j,p}h^{p^{\ast}}_{ijk}h^{p^{\ast}}_{ijl}=0, \ \ k,l=1, 2, \\
&\sum_{i,j,k}(h^{p^{\ast}}_{ijk})^{2}=S(\frac{3}{2}S+1)-2H^{2}S+\frac{1}{2}H^{4}
+\sum_{i,j,p,q}H^{p^{\ast}}h^{p^{\ast}}_{ij}H^{q^{\ast}}h^{q^{\ast}}_{ij}.
\end{aligned}
\end{equation}
If $\vec{H}=\sum_{p}H^{p^{\ast}}e_{p^{\ast}}=0$ at $p\in M^{2}$, we know $H^{2}\leq S$.
If $\vec{H}=\sum_{p}H^{p^{\ast}}e_{p^{\ast}}\neq 0$
at $p\in M^{2}$, we choose a local frame field $\{e_{1}, e_{2}\}$
such that
$$ \vec{H}=H^{1^{\ast}}e_{1^{\ast}}, \ \ H^{1^{\ast}}=|\vec{H}|=H, \ \ H^{2^{\ast}}=h^{2^{\ast}}_{11}+h^{2^{\ast}}_{22}=0.
$$
Besides, $h^{1^{\ast}}_{11}$, $h^{1^{\ast}}_{22}$ and $h^{2^{\ast}}_{11}$ by $\lambda_{1}$, $\lambda_{2}$ and $\lambda$, respectively, then
$$S=\lambda^{2}_{1}+3\lambda^{2}_{2}+4\lambda^{2}, \ \ H^{2}=(\lambda_{1}+\lambda_{2})^{2}\leq \frac{4}{3}(\lambda^{2}_{1}+3\lambda^{2}_{2})\leq \frac{4}{3}S,$$
and the equality of the above inequality holds if and only if
$$\lambda_{1}=3\lambda_{2}, \ \ \lambda=0.$$
Thus, $H^{2}$ is bounded from above since $S$ is constant. From the Proposition \ref{proposition 3.1},
we know that the Bakry-Emery Ricci curvature of $x:M^{2}\rightarrow \mathbb{C}^{2}$ is bounded from below. By applying the maximum principle
of Omori-Yau type concerning the operator $\mathcal{L}$ to the function $H^{2}$, there exists a sequence $\{p_{t}\} \in M^{2}$ such that
\begin{equation*}
\lim_{t\rightarrow\infty} H^{2}(p_{t})=\sup H^{2},\quad
\lim_{t\rightarrow\infty} |\nabla H^{2}|(p_{t})=0,\quad
\lim_{t\rightarrow\infty}\mathcal{L} H^{2}(p_{t})\leq 0.
\end{equation*}
For the constant $S$, by use of \eqref{2.1-14}, \eqref{2.1-15} and \eqref{2.1-16}, we know that
$\{h^{p^{\ast}}_{ij}(p_{t})\}$, $\{h^{p^{\ast}}_{ijk}(p_{t})\}$ and $\{h^{p^{\ast}}_{ijkl}(p_{t})\}$ are bounded sequences for $ i, j, k, l, p = 1,2$.
One can assume
\begin{equation*}
\begin{aligned}
&\lim_{t\rightarrow\infty}H^{2}(p_{t})=\sup H^{2}=\bar H^{2}, \ \ \lim_{t\rightarrow\infty}h^{p^{\ast}}_{ij}(p_{t})=\bar h^{p^{\ast}}_{ij}, \ \ \lim_{t\rightarrow\infty}h^{p^{\ast}}_{ijk}(p_{t})=\bar h^{p^{\ast}}_{ijk},  \\
&\lim_{t\rightarrow\infty}h^{p^{\ast}}_{ijkl}(p_{t})=\bar h^{p^{\ast}}_{ijkl}, \ \ i, j, k, l, p=1, 2.
\end{aligned}
\end{equation*}
Then it follows from \eqref{2.1-13} and the third equation of \eqref{3.1-1} that
\begin{equation}\label{3.1-2}
\aligned
0\geq&\lim_{t\rightarrow\infty}|\nabla^{\perp}\vec{H}|^{2}(p_{t})-\bar H^{2}-\sum_{i,j,p,q}\bar H^{p^{\ast}}\bar h^{p^{\ast}}_{ij}\bar H^{q^{\ast}}\bar h^{q^{\ast}}_{ij} \\
=&\lim_{t\rightarrow\infty}|\nabla^{\perp}\vec{H}|^{2}(p_{t})-\sum_{i,j,k}(\bar h^{p^{\ast}}_{ijk})^{2}+\frac{1}{2}(\bar H^{2}-S)(\bar H^{2}-3S-2).
\endaligned
\end{equation}

If $\sup H^{2}=0$, we know $H^{2}\equiv 0$. From now, we only consider that $\sup H^{2}>0$.
In fact, this situation does not exist.
Without loss of the generality,
at each point $p_{t}$, we can assume $H^{2}(p_{t})\neq 0$.

\noindent
Since $\lim_{t\rightarrow\infty} |\nabla H^{2}(p_{t})|=0$ and $|\nabla H^{2}|^{2}=4\sum_{k}(\sum_{p}H^{p^{\ast}}H^{p^{\ast}}_{,k})^{2}$,
we can see that
\begin{equation}\label{3.1-3}
\bar H^{1^{\ast}}_{,k}=0, \ \bar h^{1^{\ast}}_{11k}+\bar h^{1^{\ast}}_{22k}=0, \ \ k=1, 2.
\end{equation}
It follows from the first formula of \eqref{2.1-12} and $h^{2^{\ast}}_{11}+h^{2^{\ast}}_{22}=0$ that
\begin{equation}\label{3.1-4}
\begin{aligned}
&\bar H^{1^{\ast}}_{,1}=-\bar \lambda_{1}\lim_{t\rightarrow\infty} \langle x, e_{1} \rangle(p_{t})-\bar \lambda\lim_{t\rightarrow\infty} \langle x, e_{2} \rangle(p_{t}), \\
&\bar H^{1^{\ast}}_{,2}=-\bar \lambda\lim_{t\rightarrow\infty} \langle x, e_{1} \rangle(p_{t})-\bar \lambda_{2}\lim_{t\rightarrow\infty} \langle x, e_{2} \rangle(p_{t}), \\
\end{aligned}
\end{equation}
and
\begin{equation}\label{3.1-5}
\begin{aligned}
&\bar H^{2^{\ast}}_{,1}=-\bar \lambda\lim_{t\rightarrow\infty} \langle x, e_{1} \rangle(p_{t})-\bar \lambda_{2}\lim_{t\rightarrow\infty} \langle x, e_{2} \rangle(p_{t}), \\
&\bar H^{2^{\ast}}_{,2}=-\bar \lambda_{2}\lim_{t\rightarrow\infty} \langle x, e_{1} \rangle(p_{t})+\bar \lambda\lim_{t\rightarrow\infty} \langle x, e_{2} \rangle(p_{t}).
\end{aligned}
\end{equation}
By making use of the first equation of \eqref{3.1-1} and \eqref{3.1-3}, we obtain
\begin{equation}\label{3.1-6}
(\bar \lambda_{1}-3\bar \lambda_{2})\bar h^{1^{\ast}}_{11k}+3\bar \lambda\bar h^{2^{\ast}}_{11k}-\bar \lambda\bar h^{2^{\ast}}_{22k}=0, \ \ k=1, 2.
\end{equation}

By the above formulas, we can obtain the claim that
$$\bar \lambda=0.$$
In fact, we first assume that $\bar \lambda\neq0$ and then deduce
a contradiction.

\noindent
By \eqref{3.1-3}, \eqref{3.1-6} and the symmetry of indices, we infer
\begin{equation}\label{3.1-7}
(\bar \lambda_{1}-3\bar \lambda_{2})\bar h^{1^{\ast}}_{111}+4\bar \lambda\bar h^{2^{\ast}}_{111}=0, \ \
-3\bar \lambda\bar h^{1^{\ast}}_{111}-\bar \lambda\bar h^{2^{\ast}}_{222}+(\bar \lambda_{1}-3\bar \lambda_{2})\bar h^{2^{\ast}}_{111}=0.
\end{equation}
\eqref{3.1-3} and \eqref{3.1-4} imply that
\begin{equation}\label{3.1-8}
(\bar \lambda_{1}\bar \lambda_{2}-\bar \lambda^{2})\lim_{t\rightarrow\infty} \langle x, e_{1} \rangle(p_{t})=0, \ \
(\bar \lambda_{1}\bar \lambda_{2}-\bar \lambda^{2})\lim_{t\rightarrow\infty} \langle x, e_{2} \rangle(p_{t})=0.
\end{equation}
The claim is divided into two cases.

\noindent {\bf Case 1: $\bar \lambda_{1}\bar \lambda_{2}-\bar \lambda^{2}\neq 0$.}

\noindent
From \eqref{3.1-8}, we can see
\begin{equation*}
\lim_{t\rightarrow\infty} \langle x, e_{1} \rangle(p_{t})=\lim_{t\rightarrow\infty} \langle x, e_{2} \rangle(p_{t})=0,
\end{equation*}
This together with \eqref{3.1-5} imply that
\begin{equation}\label{3.1-9}
\bar H^{2^{\ast}}_{,k}=0, \ \bar h^{2^{\ast}}_{11k}+\bar h^{2^{\ast}}_{22k}=0, \ \ k=1, 2.
\end{equation}
Combining \eqref{3.1-3}, \eqref{3.1-7} with \eqref{3.1-9}, we obtain
\begin{equation*}
(\bar \lambda_{1}-3\bar \lambda_{2})\bar h^{1^{\ast}}_{111}+4\bar \lambda\bar h^{2^{\ast}}_{111}=0, \ \
-4\bar \lambda\bar h^{1^{\ast}}_{111}+(\bar \lambda_{1}-3\bar \lambda_{2})\bar h^{2^{\ast}}_{111}=0.
\end{equation*}
Thus, $$\bar h^{1^{\ast}}_{111}=\bar h^{2^{\ast}}_{111}=0$$ since $(\bar \lambda_{1}-3\bar \lambda_{2})^{2}+16\bar \lambda^{2}\neq0$.

\noindent
Then \eqref{3.1-3} and \eqref{3.1-9} give that
$$\bar h^{p^{\ast}}_{ijk}=0, \ \ i,j,k,p=1,2.$$
Consequently, by taking the limit of the third equation of \eqref{3.1-1}, we infer
\begin{equation*}
\begin{aligned}
0&=S(\frac{3}{2}S+1)-2\bar H^{2}S+\frac{1}{2}\bar H^{4}
+\sum_{i,j,p,q}\bar H^{p^{\ast}}\bar h^{p^{\ast}}_{ij}\bar H^{q^{\ast}}\bar h^{q^{\ast}}_{ij}\\
&=S(\frac{3}{2}S+1)-2\bar H^{2}S+\frac{1}{2}\bar H^{4}
+\bar H^{2}(\bar \lambda^{2}_{1}+\bar \lambda^{2}_{2}+2\bar \lambda^{2})\\
&=S(\frac{1}{2}S+1)+(S-\bar H^{2})^{2}+\frac{1}{2}\bar H^{2}(\bar \lambda_{1}-\bar \lambda_{2})^{2}
+2\bar H^{2}\bar \lambda^{2}.
\end{aligned}
\end{equation*}
Then $$S=\bar H^{2}=0.$$ It contradicts the hypothesis.

\noindent {\bf Case 2: $\bar \lambda_{1}\bar \lambda_{2}-\bar \lambda^{2}=0$.}

\noindent
Since $\bar \lambda_{1}\bar \lambda_{2}-\bar \lambda^{2}=0$, then \eqref{3.1-7} is equivalent to
\begin{equation}\label{3.1-10}
(\bar \lambda_{1}+3\bar \lambda_{2})^{2}\bar h^{1^{\ast}}_{111}=-4\bar \lambda^{2}\bar h^{2^{\ast}}_{222}, \ \
(\bar \lambda_{1}+3\bar \lambda_{2})^{2}\bar h^{2^{\ast}}_{111}=\bar \lambda(\bar \lambda_{1}-3\bar \lambda_{2})\bar h^{2^{\ast}}_{222}.
\end{equation}
Suppose $\bar \lambda_{1}+3\bar \lambda_{2}=0$, we have that $\bar h^{2^{\ast}}_{222}=0$,
and \eqref{3.1-7} yields
\begin{equation*}
-3\bar \lambda_{2}\bar h^{1^{\ast}}_{111}+2\bar \lambda\bar h^{2^{\ast}}_{111}=0, \ \
-\bar \lambda\bar h^{1^{\ast}}_{111}-2\bar \lambda_{2}\bar h^{2^{\ast}}_{111}=0.
\end{equation*}
Then  $\bar h^{1^{\ast}}_{111}=\bar h^{2^{\ast}}_{111}=0$. That is,
$$\lim_{t\rightarrow\infty}|\nabla^{\perp}\vec{H}|^{2}(p_{t})=0, \ \ \bar h^{p^{\ast}}_{ijk}=0, \ \ i,j,k,p=1,2.$$
It follows from \eqref{3.1-2} that
\begin{equation*}
(\bar H^{2}-S)(\bar H^{2}-3S-2)\leq0.
\end{equation*}
It is a contradiction since $\bar H^{2}-S=-8\bar \lambda^{2}_{2}-4\bar \lambda^{2}<0$.

\noindent
Suppose $\bar \lambda_{1}+3\bar \lambda_{2}\neq0$, \eqref{3.1-10} implies that
\begin{equation*}
\bar h^{1^{\ast}}_{111}=-\bar h^{1^{\ast}}_{221}=-\frac{4\bar \lambda^{2}}{(\bar \lambda_{1}+3\bar \lambda_{2})^{2}}\bar h^{2^{\ast}}_{222}, \ \
\bar h^{2^{\ast}}_{111}=-\bar h^{1^{\ast}}_{222}=\frac{\bar \lambda(\bar \lambda_{1}-3\bar \lambda_{2})}{(\bar \lambda_{1}+3\bar \lambda_{2})^{2}}\bar h^{2^{\ast}}_{222}.
\end{equation*}
Using \eqref{3.1-3}, \eqref{3.1-4} and \eqref{3.1-5}, by a direct caculation, we have
\begin{equation*}
\begin{aligned}
&\lim_{t\rightarrow\infty}|\nabla^{\perp}\vec{H}|^{2}(p_{t})=(\bar H^{2^{\ast}}_{,2})^{2}=(\bar h^{2^{\ast}}_{112}+\bar h^{2^{\ast}}_{222})^{2}=\frac{(\bar \lambda^{2}_{1}+9\bar \lambda^{2}_{2}+10\bar \lambda^{2})^{2}}{(\bar \lambda_{1}+3\bar \lambda_{2})^{4}}(\bar h^{2^{\ast}}_{222})^{2}, \\
&\sum_{i,j,k}(\bar h^{p^{\ast}}_{ijk})^{2}=7(\bar h^{1^{\ast}}_{111})^{2}+8(\bar h^{2^{\ast}}_{111})^{2}+(\bar h^{2^{\ast}}_{222})^{2}=\frac{(\bar \lambda^{2}_{1}+9\bar \lambda^{2}_{2}+10\bar \lambda^{2})^{2}}{(\bar \lambda_{1}+3\bar \lambda_{2})^{4}}(\bar h^{2^{\ast}}_{222})^{2}.
\end{aligned}
\end{equation*}
Then it follows from \eqref{3.1-2} that
\begin{equation*}
(\bar H^{2}-S)(\bar H^{2}-3S-2)\leq0.
\end{equation*}
It is a contradiction since $\bar H^{2}-S=-2\bar \lambda^{2}_{2}-2\bar \lambda^{2}<0$.
So the claim is proved.

Next, we now use $\bar \lambda =0$ to complete the proof of proposition \ref{proposition 3.2}.
For $\bar \lambda =0$, \eqref{3.1-4} implies
\begin{equation}\label{3.1-11}
\bar \lambda_{1}\lim_{t\rightarrow\infty} \langle x, e_{1} \rangle(p_{t})=0, \ \ \bar \lambda_{2}\lim_{t\rightarrow\infty} \langle x, e_{2} \rangle(p_{t})=0.
\end{equation}
If $\bar \lambda_{1}=0$, we know that $S=3\bar H^{2}=3\bar \lambda^{2}_{2}$ and  $\bar \lambda_{2}\neq0$ since $\bar H^{2}\neq0$.

\noindent
By making use of \eqref{3.1-3} and \eqref{3.1-6}, we infer
$$
\bar h^{1^{\ast}}_{11k}=\bar h^{1^{\ast}}_{22k}=0, \ \ k=1,2.
$$
Namely,
$$
\bar h^{1^{\ast}}_{111}=\bar h^{1^{\ast}}_{112}=\bar h^{1^{\ast}}_{221}=\bar h^{1^{\ast}}_{222}=0,
$$
and
$$
\lim_{t\rightarrow\infty}|\nabla^{\perp}\vec{H}|^{2}(p_{t})=\sum_{i,j,k}(\bar h^{p^{\ast}}_{ijk})^{2}=(\bar h^{2^{\ast}}_{222})^{2}.
$$
Then it follows from \eqref{3.1-2} that
\begin{equation*}
(\bar H^{2}-S)(\bar H^{2}-3S-2)\leq0.
\end{equation*}
It is a contradiction since $\bar H^{2}-S=-2\bar \lambda^{2}_{2}<0$.

\noindent
If $\bar \lambda_{2}=0$, we infer that $S=\bar H^{2}=\bar \lambda^{2}_{1}$ and  $\bar \lambda_{1}\neq0$ since $\bar H^{2}\neq0$.

\noindent
From \eqref{3.1-3}, \eqref{3.1-5} and \eqref{3.1-6}, we get
$$\bar H^{2^{\ast}}_{,1}=\bar H^{2^{\ast}}_{,2}=0, \ \ \bar h^{1^{\ast}}_{11k}=\bar h^{1^{\ast}}_{22k}=0, \ \ k=1,2.$$
Then we can draw a conclusion that $$\bar h^{p^{\ast}}_{ijk}=0, \ \ i,j,k,p=1,2.$$
By taking the limit of the third equation of \eqref{3.1-1}, we infer
\begin{equation*}
\begin{aligned}
0&=S(\frac{3}{2}S+1)-2\bar H^{2}S+\frac{1}{2}\bar H^{4}
+\sum_{i,j,p,q}\bar H^{p^{\ast}}\bar h^{p^{\ast}}_{ij}\bar H^{q^{\ast}}\bar h^{q^{\ast}}_{ij}\\
&=S(\frac{3}{2}S+1)-2\bar H^{2}S+\frac{1}{2}\bar H^{4}
+\bar H^{2}\bar \lambda^{2}_{1}\\
&=S(S+1).
\end{aligned}
\end{equation*}
Thus, $S=\bar H^{2}=0$.
It contradicts the hypothesis.

\noindent
If $\bar \lambda_{1}\bar \lambda_{2}\neq0$, it is easy to draw that by \eqref{3.1-11}
\begin{equation}\label{3.1-12}
\lim_{t\rightarrow\infty} \langle x, e_{1} \rangle(p_{t})=0, \ \ \lim_{t\rightarrow\infty} \langle x, e_{2} \rangle(p_{t})=0.
\end{equation}
By making use of \eqref{3.1-3}, \eqref{3.1-5}, \eqref{3.1-6} and \eqref{3.1-12},
we obtain
\begin{equation}\label{3.1-13}
\bar H^{p^{\ast}}_{,k}=\bar h^{p^{\ast}}_{11k}+\bar h^{p^{\ast}}_{22k}=0, \ \ k,p=1,2,
\end{equation}
and
\begin{equation}\label{3.1-14}
(\bar \lambda_{1}-3\bar \lambda_{2})\bar h^{1^{\ast}}_{11k}=0, \ \ k=1,2.
\end{equation}
Suppose $\bar \lambda_{1}-3\bar \lambda_{2}\neq0$, from \eqref{3.1-13} and \eqref{3.1-14}, we have
$$\bar h^{p^{\ast}}_{ijk}=0, \ \ i,j,k,p=1,2.$$
By taking the limit of the third equation of \eqref{3.1-1}, we infer
\begin{equation*}
\begin{aligned}
0&=S(\frac{3}{2}S+1)-2\bar H^{2}S+\frac{1}{2}\bar H^{4}
+\sum_{i,j,p,q}\bar H^{p^{\ast}}\bar h^{p^{\ast}}_{ij}\bar H^{q^{\ast}}\bar h^{q^{\ast}}_{ij}\\
&=S(\frac{3}{2}S+1)-2\bar H^{2}S+\frac{1}{2}\bar H^{4}
+\bar H^{2}(\bar \lambda^{2}_{1}+\bar \lambda^{2}_{2})\\
&=S(\frac{1}{2}S+1)+(S-\bar H^{2})^{2}+\frac{1}{2}\bar H^{2}(\bar \lambda_{1}-\bar \lambda_{2})^{2}.
\end{aligned}
\end{equation*}
Then $S=\bar H^{2}=0$. It contradicts the hypothesis.

\noindent
Suppose $\bar \lambda_{1}-3\bar \lambda_{2}=0$, we know that
\begin{equation}\label{3.1-15}
\bar H=\frac{4}{3}\bar \lambda_{1}, \ \ \bar H^{2}=\frac{16}{9}\bar \lambda^{2}_{1}, \ \ S=\frac{4}{3}\bar \lambda^{2}_{1}, \ \ \bar H^{2}=\frac{4}{3}S.
\end{equation}
Taking the limit of  the third equation of \eqref{3.1-1} and by a direct caculation, we can see that
\begin{equation*}
\sum_{i,j,k,p}(\bar h^{p^{\ast}}_{ijk})^{2}=\frac{5}{6}S^{2}+S.
\end{equation*}
Besides, using \eqref{3.1-13} and the symmetry of indices, we infer
\begin{equation*}
\sum_{i,j,p}(\bar h^{p^{\ast}}_{ij1})^{2}=4\Big((\bar h^{1^{\ast}}_{111})^{2}+(\bar h^{1^{\ast}}_{112})^{2}\Big), \ \
\sum_{i,j,k,p}(\bar h^{p^{\ast}}_{ijk})^{2}=8\Big((\bar h^{1^{\ast}}_{111})^{2}+(\bar h^{1^{\ast}}_{112})^{2}\Big).
\end{equation*}
Then
\begin{equation}\label{3.1-16}
\sum_{i,j,p}(\bar h^{p^{\ast}}_{ij1})^{2}=\frac{1}{2}\sum_{i,j,k,p}(\bar h^{p^{\ast}}_{ijk})^{2}=\frac{5}{12}S^{2}+\frac{1}{2}S.
\end{equation}
Taking the limit of the second equation of \eqref{3.1-1} and choosing $k=l=1$, we obtain
\begin{equation*}
\bar h^{1^{\ast}}_{11}\bar h^{1^{\ast}}_{1111}+3\bar h^{1^{\ast}}_{22}\bar h^{1^{\ast}}_{2211}
+3\bar h^{2^{\ast}}_{11}\bar h^{2^{\ast}}_{1111}+\bar h^{2^{\ast}}_{22}\bar h^{2^{\ast}}_{2222}=-\sum_{i,j,p}(\bar h^{p^{\ast}}_{ij1})^{2},
\end{equation*}
namely,
\begin{equation}\label{3.1-17}
\bar \lambda_{1}\bar h^{1^{\ast}}_{1111}+3\bar \lambda_{2}\bar h^{1^{\ast}}_{2211}
=-\sum_{i,j,p}(\bar h^{p^{\ast}}_{ij1})^{2}.
\end{equation}
It follows from $\bar \lambda_{1}=3\bar \lambda_{2}$, \eqref{3.1-16} and \eqref{3.1-17} that
\begin{equation}\label{3.1-18}
\bar \lambda_{1}(\bar h^{1^{\ast}}_{1111}+\bar h^{1^{\ast}}_{2211})
=-(\frac{5}{12}S^{2}+\frac{1}{2}S).
\end{equation}
By mean of the second equation of \eqref{2.1-12}, \eqref{3.1-12} and \eqref{3.1-15} and choosing $i=j=p=1$, we know
\begin{equation*}
\begin{aligned}
\bar H^{1^{\ast}}_{,11}
=&-\bar h^{1^{\ast}}_{11}-\sum_{k}\bar h^{1^{\ast}}_{1k}\bar h^{1^{\ast}}_{k1}\bar H^{1^{\ast}}\\
=&-\bar \lambda_{1}-\bar \lambda^{2}_{1}\bar H=-\frac{3}{4}\bar H(S+1),
\end{aligned}
\end{equation*}
which implies that
\begin{equation}\label{3.1-19}
\bar h^{1^{\ast}}_{1111}+\bar h^{1^{\ast}}_{2211}=-\frac{3}{4}\bar H(S+1),
\end{equation}
Using \eqref{3.1-15}, \eqref{3.1-18} and \eqref{3.1-19}, we infer
\begin{equation*}
-\frac{3}{4}\bar \lambda_{1}\bar H(S+1)=-\frac{3}{4}S(S+1)=-(\frac{5}{12}S^{2}+\frac{1}{2}S).
\end{equation*}
Then
$S=\bar H^{2}=0$.
It contradicts the hypothesis.
\end{proof}

\vskip3mm
\noindent
{\it Proof of Theorem \ref{theorem 1.2}}.
If $S$ is constant,
from the Proposition \ref{proposition 3.2}, we have that $H^{2}\equiv0$ which implies that $M^{2}$ is a hyperplane $\mathbb{R}^{2}$ through the origin. In fact, it follows from the definition of lagrangian self-expander and the first equation of \eqref{2.1-12} that
$$H^{p^{\ast}}=\langle x, e_{p^{\ast}}\rangle=0, \ \ \sum_{k}h^{p^{\ast}}_{ik}\langle x, e_{k}\rangle=0, \ \ i,p=1,2.$$
That is,
$$h^{1^{\ast}}_{11}+h^{1^{\ast}}_{22}=0, \ \ h^{2^{\ast}}_{11}+h^{2^{\ast}}_{22}=0$$
and
\begin{equation*}
\begin{aligned}
&h^{1^{\ast}}_{11}\langle x, e_{1}\rangle+h^{1^{\ast}}_{12}\langle x, e_{2}\rangle=0, \ \
&h^{2^{\ast}}_{21}\langle x, e_{1}\rangle+h^{2^{\ast}}_{22}\langle x, e_{2}\rangle=0, \\
&h^{1^{\ast}}_{21}\langle x, e_{1}\rangle+h^{1^{\ast}}_{22}\langle x, e_{2}\rangle=0, \ \
&h^{2^{\ast}}_{11}\langle x, e_{1}\rangle+h^{2^{\ast}}_{12}\langle x, e_{2}\rangle=0.
\end{aligned}
\end{equation*}
Then by the symmetry of indices, we infer
$$\Big((h^{1^{\ast}}_{11})^{2}+(h^{2^{\ast}}_{11})^{2}\Big)\langle x, e_{k}\rangle=0, \ \ k=1,2.$$
Suppose $\langle x, e_{k}\rangle=0$, we draw
$$0=\langle x, e_{k}\rangle_{,k}=1+\sum_{p}h^{p^{\ast}}_{kk}\langle x, e_{p^{\ast}}\rangle=1.$$
It is impossible. Thus,
$$(h^{1^{\ast}}_{11})^{2}+(h^{2^{\ast}}_{11})^{2}=0,$$
which implies $$ h^{p^{\ast}}_{ik}=0, \ \ i,k,p=1,2.$$
So the main theorem of the present paper is proved.

\begin{flushright}
$\square$
\end{flushright}

\noindent{\bf Acknowledgements.}
The first author was partially supported by the China Postdoctoral Science Foundation Grant No.2022M711074. The second author was partly supported by grant No.12171164 of NSFC, GDUPS (2018), Guangdong Natural Science Foundation Grant No.2023A1515010510.

\end{document}